\documentclass[11pt]{amsart}
\usepackage[margin=1in]{geometry}

\usepackage{todonotes, xspace, tikz, bbm, physics, mathrsfs, shuffle, csquotes, graphicx, mathtools}
\usepackage[colorlinks=true,citecolor=blue, linkcolor=red]{hyperref}

\DeclareMathOperator{\ept}{ept}

\newcommand{\bE}{\mathbb E}

\newcommand{\eps}{\varepsilon}

\newcommand{\bigp}[1]{\left( #1 \right)} 
\newcommand{\bigb}[1]{\left[ #1 \right]} 

\newcommand{\Var}[1]{\mathrm{Var}\left[ #1 \right]}

\newtheorem{counter}{Theorem}[section]
\newtheorem{theorem}[counter]{Theorem}

\newtheorem{lemma}[counter]{Lemma}

\newtheorem{remark}[counter]{Remark}

\newcommand{\floor}[1]{\left\lfloor #1 \right\rfloor}

\title{Probabilistic zero forcing on grid, regular, and hypercube graphs}
\author{David Hu and Alec Sun}

\begin{document}

\maketitle

\begin{abstract}
Probabilistic zero-forcing is a coloring process on a graph. In this process, an initial set of vertices is colored blue, and the remaining vertices are colored white. At each time step, blue vertices have a non-zero probability of forcing white neighbors to blue. The expected propagation time is the expected amount of time needed for every vertex to be colored blue. We derive asymptotic bounds for the expected propagation time of several families of graphs. We prove the optimal asymptotic bound of $\Theta(m+n)$ for $m\times n$ grid graphs. We prove an upper bound of $O\bigp{\frac{\log d}{d} \cdot n}$ for $d$-regular graphs on $n$ vertices and provide a graph construction that exhibits a lower bound of $\Omega\bigp{\frac{\log \log d}{d} \cdot n}$. Finally, we prove an asymptotic upper bound of $O(n \log n)$ for hypercube graphs on $2^n$ vertices.
\end{abstract}

\section{Introduction}

Zero-forcing is a widely studied coloring process on a graph. Initially, some vertices in a graph $G$ are colored blue, while other vertices are white. At each time step, each blue vertex $u$ connected to a white vertex $v$ changes the color of $v$ to blue if $v$ is the only white neighbor of $u$. When this happens, we say that $u\ \textit{forces}\ v$. At a given time step, a white vertex $v$ may be forced by one or more blue vertices, in which case it will become blue. Every vertex that is colored blue will always remain blue, and it may force other white vertices to blue in future time steps. The concept of zero-forcing has been used to attack the maximum nullity problem of combinatorial matrix theory in \cite{1}, \cite{2}, \cite{7}, and \cite{12}. Zero forcing is also related to power domination \cite{3} and graph searching \cite{15}. Viewing zero forcing through the lens of dynamical processes, Fallat et al. \cite{8} and Hogben et al. \cite{11} have studied the number of steps it takes for an initial vertex set to force all other vertices to blue, assuming that all the vertices will eventually become blue under the zero forcing rule. This is known as the \emph{propagation time} of a set. Zero forcing can potentially model certain real world propagation processes such a rumor spreading. That being said, the deterministic nature of zero-forcing presents an obstacle for simulating the seemingly random process of rumor spreading in the real world.

The probabilistic color change rule is a probabilistic modification of the classic zero-forcing coloring rule introduced by Kang and Yi \cite{13}. For every blue vertex $u$ connected to a white vertex $v$, $u$ forces $v$ with probability 
$$\Pr[u \rightarrow v] = \frac{C[u]}{\deg u},$$
where $C[u]$ denotes the number of blue vertices in the closed neighborhood of $u$ including $u$ itself, and $\deg u$ is the total number of vertices connected to $u$. When some blue vertex has exactly one white neighbor, note that the probabilistic color change rule corresponds to the classical color change rule because that white neighbor is forced blue with probability 1. For a random probabilistic zero-forcing process on a graph $G$ which initially starts with a set of vertices $S$ colored blue, the $\emph{propagation time}$ of $S$, taking values in $\mathbb{N}\cup\{\infty\},$ is defined as the number of time steps until all vertices in $G$ are colored blue. The $\emph{expected propagation time}$ of $S$, denoted by $\ept(G,S)$, is the expected value of the propagation time of $S$. We also define the expected propagation time of a graph $G$, denoted by $\ept(G)$, as the minimum expected propagation time over all single-vertex subsets $S= \{v\}, v\in G$.

As noted in \cite{english}, probabilistic zero-forcing is very similar to the well-studied \emph{push} and \emph{pull} models for rumor spreading from theoretical computer science \cite{5e, 13e}. For the \emph{push} model, one starts with a set of blue vertices, and at each time step, each blue vertex chooses one neighbor independently and uniformly at random and forces that vertex blue, if that vertex is white. For the \emph{pull} model, at each time step each white vertex chooses a neighbor independently and uniformly at random, and the white vertex turns blue if the chosen neighbor is blue. The two models can also be combined to create a \emph{push and pull} model in which at each time step, blue vertices choose a random neighbor to force and white vertices choose a random neighbor to try to become blue. As with probabilistic zero-forcing, the primary parameter of interest is the expected propagation time.

Returning now to probabilistic zero forcing, recent work has been done to compute bounds on propagation time for many families of graphs, including paths, cycles, complete graphs, and bipartite graphs \cite{Finite, GH}. In \cite{Sun}, it is shown for a connected graph $G$ with order $n$ and radius $r$ that $\ept(G) \le O\bigp{r \log \frac{n}{r}}$, and the authors construct an example to show tightness of the asymptotic. They also prove for a connected graph $G$ with order $n$ that $\ept(G)\le \frac{n}{2} + o(n)$. Proving further bounds for expected propagation times of different families of graphs has been proposed as an area of study in \cite{GH} and \cite{Sun}. In \cite{english}, the authors have established high probability results for the expected propagation time for the Erd\H{o}s-Renyi graph $G(n,p)$, where $p$ is a function of $n$.

In this paper, we prove bounds on the expected propagation times of several other well-known families of graphs. Our main results are as follows:

\begin{theorem} \label{cityblock} We have $$\bigp{\frac12 - o(1)}(m+n) \le \ept(G_{m\times n}) \le (4+o(1))(m+n).$$
\end{theorem}

\begin{theorem} \label{dregular} The expected propagation time of a $d$-regular graph $G$ is $O\bigp{n\cdot \frac{\log d}{d}}$ for $d\ge 2$. Furthermore, there exists a family of $d$-regular graphs, with $d$ constant as a function of $n$, such that the expected propagation time is $\Omega\bigp{n \cdot \frac{\log\log d}{d}}$.
\end{theorem}
\begin{theorem} \label{hypercube} The expected propagation time of an $n$-dimensional hypercube graph with $2^n$ vertices is $O(n \log n)$. 
\end{theorem}

\section{Preliminaries}

We review some well-known tools from probability theory that will be used in the paper.

\begin{theorem}[Chebyshev's Inequality]
Given a random variable $X$, for all $\lambda \ge 0$ we have $$\Pr\bigb{|X - \mathbb{E}[X]| \ge \lambda} \le \frac{\Var{X}}{\lambda^2}.$$
\end{theorem}


\begin{theorem}[Chernoff Bound]
Let $X_1, X_2, \ldots, X_n$ be independent random variables taking on values in $\{0,1\}$. Let $X = \displaystyle \sum_{i = 1}^{n} X_i$ and $\mu = \mathbb{E}[X]$. Then $$\Pr[X \le (1 - \delta) \mu] \le e^{-\frac12\mu \delta^2}.$$
\end{theorem}

\begin{theorem}[Edge Isoperimetric Inequality for Hypercubes] Let $S$ be any set of vertices in a dimension $n$ hypercube graph with $2^n$ vertices. Then the number of edges in $G$ between a vertex in $S$ and a vertex not in $S$ is at least $|S|(n - \log_2{|S|})$.
\end{theorem}

We will also use a coupling argument proven in \cite{Sun}.

\begin{lemma}[\cite{Sun}, Lemma 2.14]\label{okay}

Suppose that initially in some graph $G$, some set of vertices $S$ is colored blue. We follow a modified probabilistic process where at the $t$-th point in time, for any connected blue vertex $u$ and white vertex $v$, the probability $\Pr_t[u \rightarrow v]$ that $u$ converts $v$ to become blue at the $t$-th step, is some function of $G,u,v$, and $B_{t-1}$, the set of blue vertices after the $(t-1)$-th step. In addition, suppose that for all $u, v$, $\Pr_t[u \rightarrow v] \le \frac{C[u]}{\deg u}$. Then the expected propagation time of $S$ under this modified probabilistic color change rule is less than or equal to the expected propagation time of $S$ under the original probabilistic color change rule.
\end{lemma}

\section{Grid graphs}

In this section we prove Theorem \ref{cityblock}. For a grid graph $G_{m\times n}$ on $m\times n$ vertices,  given any initial vertex, one of the four corner vertices of $G_{m,n}$ is a distance of at least $\frac{1}{2}(m+n-2)$ away from the initial vertex. Hence it must take at least this amount of time to color it blue, and the lower bound $\ept(G_{m\times n}) \ge \bigp{\frac{1}{2}-o(1)}(m+n)$ follows immediately. Thus it suffices to show $\ept(G_{m\times n}) \le (4+o(1))(m+n)$.

We consider a modified color change rule, in which a vertex that is currently white and has a blue neighbor becomes blue with probability exactly $\frac14$. Note that each vertex $v$ in $G_{m\times n}$ has $\deg v\le 4$, so the probability of a blue vertex forcing any adjacent white vertex is at least $\frac14$. By Lemma \ref{okay}, the expected propagation time under this modified process, which we denote by $\ept'$, is at least as large as that of the original process.

\begin{lemma}\label{corner-lemma-grid}
We have $\ept'(G_{m\times n}, \{v_\text{corner}\})\le (4+o(1))(m+n)$ for any of the four corner vertices $v_\text{corner}\in G$.
\end{lemma}
 
\begin{proof}
Without loss of generality, assume $m \le n$. Assign Cartesian coordinates $(i, j)\in [0,m-1]\times [0,n-1]$ to the vertices of $G_{m\times n}$. Without loss of generality, suppose that $v_\text{corner} = (0, 0)$. Starting from $v_\text{corner} = (0, 0)$, let $T_1$ denote the time at which all vertices on the $x$-axis are blue. Starting from a configuration in which all vertices on the $x$-axis are blue, let $T_2$ denote the amount of time it takes for all vertices to be blue. Note that $\ept'(G) \le \mathbb{E}[T_1] + \mathbb{E}[T_2]$. 

The expected time needed for all vertices on the $x$-axis to be blue is at most the expected propagation time for a modified probabilistic color change process on a line, starting with an end vertex blue and such that every blue vertex forces a white neighbor with probability $\frac14$. Denote by $t_i$ the expected time needed from when $(i,0)$ is colored blue to the time when $(i+1,0)$ is colored blue, so that $\mathbb{E}[T_1] \le \displaystyle \sum_{i=0}^{m-1} t_i.$ Using the relationship between the number of trials and expected amount of time to success, the expected number of trials from when $(i,0)$ is colored blue to the time when $(i+1,0)$ is at most $\bigp{\frac{1}{4}}^{-1} = 4$. We conclude that $t_i \leq 4$ for all $i$, so $\mathbb{E}[T_1] \le 4m - 4$.

Now consider the point in time at which the entire $x$-axis is blue. At this point, by Lemma \ref{okay}, $\mathbb{E}[T_2]$ is bounded above by the expected time it takes to color $m$ independent paths of length $n$ blue, starting from an end vertex of each path. Fix any constant $\eps > 0$. We claim that with exponentially small probability in $n$, the propagation time on any one path is more than $(4+4\eps)n$. The probability that the propagation time on a path of length $n$ is more than $(4+4\eps)n$ is at most the probability that the sum of $(4+4\eps)n$ independent $\text{Bern}(\frac14)$ random variables is at most $n$. By the Chernoff Bound for $\mu=(1+\eps)n$ and $\delta = \eps/(1+\eps)$, we have that this is bounded above by $e^{-\frac12 \mu \delta^2} = e^{-\frac{\eps^2}{2(1+\eps)}n}$. By a union bound over all $m\le n$ paths, the probability that not all paths are completely blue after $(1+\eps)n$ steps is at most $me^{-\frac{\eps^2}{2(1+\eps)}n}$. Using the relationship between the number of trials and expected amount of time to success, we have that for every $\eps>0$, $$\mathbb{E}[T_2] \le \frac{(4+4\eps)n}{1-me^{-\frac{\eps^2}{2(1+\eps)}n}} \le (4+5\eps)n$$ for sufficiently large $n$. We conclude that $$\ept'(G_{m\times n}) \le \bE[T_1] + \bE[T_2] \le (4+o(1)))(m+n).$$

\end{proof}

Note that Lemma \ref{corner-lemma-grid} implies Theorem \ref{cityblock} because $\ept'(G_{m\times n},\{v\}) \le \ept(G, \{v\})\le \ept(G)$.

For each $2 \le m, n \le 14$, we run a program to simulate the probabilistic zero-forcing process $1000$ times on an $m \times n$ grid graph. Figure \ref{figure-1} shows the average propagation time for $1000$ trials for each pair $(m,n)$ with $2\le m,n\le 14$.

\begin{figure}[h]
    \centering
    \begin{tabular}{|*{15}{c|}}
    
\hline
ept & 2 & 3 & 4 & 5 & 6 & 7 & 8 & 9 & 10 & 11 & 12 & 13 & 14 \\ \hline 
2 & 2.33 & 3.15 & 3.90 & 4.19 & 4.90 & 5.24 & 5.90 & 6.17 & 6.90 & 7.17 & 7.94 & 8.21 & 8.94\\ \hline 
3 & 3.15 & 3.89 & 4.51 & 4.92 & 5.52 & 5.95 & 6.56 & 6.88 & 7.57 & 7.87 & 8.57 & 8.88 & 9.58\\ \hline 
4 & 3.90 & 4.51 & 5.25 & 5.56 & 6.29 & 6.58 & 7.30 & 7.63 & 8.32 & 8.51 & 9.29 & 9.55 & 10.39\\ \hline 
5 & 4.19 & 4.92 & 5.56 & 5.94 & 6.71 & 6.95 & 7.63 & 7.91 & 8.51 & 8.94 & 9.64 & 9.94 & 10.64\\ \hline 
6 & 4.90 & 5.52 & 6.29 & 6.71 & 7.34 & 7.62 & 8.41 & 8.64 & 9.36 & 9.67 & 10.36 & 10.72 & 11.35\\ \hline 
7 & 5.24 & 5.95 & 6.58 & 6.95 & 7.62 & 8.00 & 8.67 & 8.97 & 9.68 & 10.00 & 10.70 & 11.03 & 11.69\\ \hline 
8 & 5.90 & 6.56 & 7.30 & 7.63 & 8.41 & 8.67 & 9.34 & 9.64 & 10.33 & 10.51 & 11.42 & 11.67 & 12.43\\ \hline 
9 & 6.17 & 6.88 & 7.63 & 7.91 & 8.64 & 8.97 & 9.64 & 9.97 & 10.66 & 10.93 & 11.67 & 11.96 & 12.63\\ \hline 
10 & 6.90 & 7.57 & 8.32 & 8.51 & 9.36 & 9.68 & 10.33 & 10.66 & 11.39 & 11.67 & 12.45 & 12.61 & 13.45\\ \hline 
11 & 7.17 & 7.87 & 8.51 & 8.94 & 9.67 & 10.00 & 10.51 & 10.93 & 11.67 & 12.03 & 12.74 & 13.04 & 13.62\\ \hline 
12 & 7.94 & 8.57 & 9.29 & 9.64 & 10.36 & 10.70 & 11.42 & 11.67 & 12.45 & 12.74 & 13.41 & 13.69 & 14.41\\ \hline 
13 & 8.21 & 8.88 & 9.55 & 9.94 & 10.72 & 11.03 & 11.67 & 11.96 & 12.61 & 13.04 & 13.69 & 14.01 & 14.72\\ \hline 
14 & 8.94 & 9.58 & 10.39 & 10.64 & 11.35 & 11.69 & 12.43 & 12.63 & 13.45 & 13.62 & 14.41 & 14.72 & 15.35\\ \hline 

\end{tabular}
    \caption{Simulated expected propagation times of $m \times n$ grid graphs for $2 \le m, n \le 14$.}
    \label{figure-1}
\end{figure}

\begin{remark} 
Experimentally, for the family of $m\times n$ rectangular graphs, $\ept(G_{m \times n})$ appears to grow asymptotically as $\frac12(m+n)$.
\end{remark}

\section{Regular graphs}

In this section, we prove Theorem \ref{dregular}. Let $k$ denote the diameter of $G$. Suppose that $u$ and $v$ are two vertices in $G$ with shortest distance equal to $k$. Let $S_i$ be the set of vertices with distance exactly $i$ from $u$, so that $S_0 = {u}$, $v \in S_k$, and $S_i = \emptyset$ for $i>k$.

\begin{lemma} For a $d$-regular graph $G$ on $n$ vertices, $k\le \frac{3n}{d + 1}$.
\end{lemma}

\begin{proof} For all $0 < i < k$, $|S_{i - 1}| + |S_i| + |S_{i+1}| > d$, since $S_i$ is nonempty and any vertex in $S_i$ has degree $d$ but can only be connected to vertices in $S_{i-1}$, $S_i$, or $S_{i+1}$. Thus

\begin{equation*} 
\begin{split}
n = & \sum_{i = 0}^{\floor{\frac{k}{3}}} \bigp{|S_{3i}| + |S_{3i+1}| + |S_{3i+2}|} \ge \bigp{d+1}\bigp{{\floor{\frac{k}{3}}} + 1} \geq \bigp{d+1}\frac{k}{3},
\end{split}
\end{equation*}
implying $$k \le \frac{3n}{d + 1}.$$

\end{proof}

\begin{lemma}  Let $H$ be a star graph with $n$ leaves, with its center blue and all other vertices colored white. If $t > 12 \log(n + 1)$, the blue vertex will propagate to all leaves in $t$ steps with probability at least $1 - (0.97)^t$.
\end{lemma}

\begin{proof} 
The proof is identical to that of Lemma 3.4 in \cite{Sun}, noting that the constants $C_1, C_2, C$, and $\alpha$ that appear there satisfy $C_1 = C_2 = 0.56 >  \log_{6} e$, $C = 10(C_1 + C_2) = 11.2$ and $\alpha = (\frac{2}{e})^{\frac{1}{10}} < 0.97$.
\end{proof}

We will now show the upper bound on expected propagation time for a $d$-regular graph. More precisely, we prove the below lemma.

\begin{lemma}\label{dregular1}
For a $d$-regular graph $G$ on $n$ vertices, $\ept(G) \le 60 \frac{\log(\frac{(d+1)}{3})}{d + 1}\cdot n + O(\log n)$.
\end{lemma}

\begin{proof}
Let $G$ be a $d$-regular graph and choose some vertex $v$. Let $w \neq v$ be some vertex such that the shortest path from $v$ to $w$ has length $s$. By carrying out the computations in Lemma 3.5 of \cite{Sun}, letting $\alpha = 0.97, C = 11.2, \beta = 0.985, C' = 2.4, C_2 = 20$, if there exists a vertex $u$ such that all vertices in $G$ are at distance at most $r$ from $u$, $\ept(G) \le 20(r \log \frac{n}{r}) + O(\log n)$. Taking $r = \frac{3n}{d+1}$, we have
\begin{align*}
    \ept(G) &\le 20 \frac{3}{d+1} \log(\frac{d+1}{3}) \cdot n + O(\log n)
    \\&\le 60\frac{\log(\frac{d+1}{3})}{d + 1} \cdot n+ O(\log n)
\end{align*}

\end{proof}

Now we construct a family of $d$-regular graphs with expected propagation time $\Omega\bigp{\frac{\log \log d}{d}\cdot n}.$ We assume that $d \geq 5$.

\begin{lemma}\label{dregular2}

Assume that $d+1 \mid n$. Start with $\frac{n}{d+1}$ copies $C_1, C_2, ..., C_{\frac{n}{d+1}}$ of $K_{d+1}$. For each copy $C_i$, designate two distinct vertices $v_{(i, 1)}, v_{(i, 2)} \in C_i$ and delete the edge connecting $v_{(i, 1)}$ and $v_{(i, 2)}$. Then, for all $1 \le i \le {\frac{n}{d+1}}$, insert an edge connecting $v_{(i, 1)}$ and $v_{(i + 1, 2)}$, where indices are taken mod $\frac{n}{d+1}$. Denote by $G$ the resulting graph. Then $\ept(G) = \Omega\bigp{\frac{\log \log d}{d}\cdot n}$.

\end{lemma}

To prove this theorem, we begin with the following lemma.

\begin{lemma}\label{complete} Suppose that for some $i$, $v_{(i-1,1)}$ and $v_{(i, 2)}$ are blue and $v_{(i, 1)}$ and $v_{(i + 1, 2)}$ are white. Then, given that $v_{(i + 1, 2)}$ remains white through the entire process, the expected amount of time for vertex $v_{(i, 1)}$ to be blue is $\Omega(\log \log d)$.\end{lemma}

\begin{proof} Let $p(b)$ be the probability, starting from a state in which $v_{(i + 1, 1)}$ is white and there are $b$ blue vertices in $C_i$, including $v_{(i, 2)}$ but not $v_{(i, 1)}$, that in one time step at most $2b^2$ additional vertices are colored blue and $v_{(i, 1)}$ remains white.

Each white vertex $v\in C_i, v\neq v_{(i,1)}$ has an independent probability of $$\bigp{1 - \frac{b}{d}}^{b - 1}\bigp{1 - \frac{b + 1}{d}} \ge \bigp{1 - \frac{b + 1}{d}}^{b}$$ of remaining white in one unit of time, where the factor $1-\frac{b+1}{d}$ comes from the fact that $v_{(i,2)}$ is connected to an additional blue vertex $v_{(i-1,1)}\in C_{i-1}$. Therefore $$p(b) \ge \bigp{1 - \frac{b}{d}}^{b-1}q(b),$$ where the factor $\bigp{1-\frac{b}{d}}^{b-1}$ comes from the fact that each of the $b-1$ blue vertices $v\in C_i, v\neq v_{i,2}$ has a probability $\frac{b}{d}$ of forcing $v_{i,1}$ blue, and $q(b)$ is the probability that, among $d - b$ independent events that happen with probability at least $1 - \bigp{1 - \frac{b + 1}{d}}^b$, less than $2b^2$ events occur. Denote a random variable $X$ for the number of such events that occur. For $d \ge 2, 1 \le b\le \sqrt[4]{d}$ we have 

\begin{equation*} 
\begin{split} 
\mathbb{E}[X] &= \bigp{d - b}\bigp{1 - \bigp{1 - \frac{b + 1}{d}}^{b}} \\ & \le (d-b)\bigp{1 - \bigp{1 - \frac{b(b + 1)}{d}}} \\ & =  (d-b)\bigp{\frac{b(b + 1)}{d}} \\ & \le b(b+1).
\end{split}
\end{equation*}
Furthermore,
\begin{equation*} 
\begin{split} 
\Var X &= (d - b)\bigp{1 - \bigp{1 - \frac{b + 1}{d}}^{b}}\bigp{1 - \frac{b + 1}{d}}^{b}
\end{split} 
\end{equation*} 
So by Chebyshev's inequality, 
\begin{equation*}
\begin{split}
1 - q(b) & = \Pr[X \ge 2b^2] \\ & \le \Pr[|X - \mathbb{E}[X]| \ge 2b^2 - b(b+1)] \\ & \le \Pr[|X - \mathbb{E}[X]| \ge b^2-b] \\ & \le \frac{\Var X}{(b^2-b)^2} \\ & = \frac{(d - b)\bigp{1 - \bigp{1 - \frac{b + 1}{d}}^{b}}\bigp{1 - \frac{b + 1}{d}}^{b}}{(b^2-b)^2} \\ & \le \frac{(d - b)\bigp{\frac{b(b+1)}{d}}\bigp{1 - \frac{b + 1}{d}}^{b}}{(b^2-b)^2} \\ & \le \frac{b(b+1)}{(b^2-b)^2} \\ & = \frac{b+1}{b(b-1)^2}
\end{split}
\end{equation*}
so
\begin{equation*}
    \begin{split}
        p(b) & \ge \bigp{1-\frac{b}{d}}^{b - 1}\bigp{1 - \frac{b+1}{b(b-1)^2}} \\ & \ge \bigp{1 - \frac{b(b-1)}{d}}\bigp{1 - \frac{b+1}{b(b-1)^2}} \\ & \ge 1 - \frac{b^2}{d} - \frac{2b}{b(b-1)^2} \\ & \ge 1 - \frac{b^2}{b^4} - \frac{8}{b^2} \\ & \ge 1  - \frac{9}{b^2}.
    \end{split}
\end{equation*}
This means that when $b \le \sqrt[4]{d}$, we have $p(b) \ge 1 - \frac{9}{b^2}$, so $1 - p(b) \leq \frac{9}{b^2}$.

We now use a similar argument as in Proposition 2.8 of \cite{GH}. Since starting with $\log{d} \leq b \leq \sqrt[4]{d}$ blue vertices and coloring at most $4b^2$ additional vertices blue means there are at most $5b^2\le b^3$ blue vertices after the round for $b \ge 5$, the probability that there are at most $b^{(3^r)}$ blue vertices after $r$ rounds is at least $\left(1-O(\frac 1 {(\log{d})^2})\right)^r$\!. Thus going from $b\le \log{d}$ blue vertices to $\sqrt[4]{d}$ blue vertices requires that $(\log{d})^{(3^r)} \ge \sqrt[4]{d}$, or $r\ge \frac14 \log_{3}\!\left(\frac {\log{d}} {\log{\log{d}}}\right) = \Omega(\log \log d)$. Thus, starting from a state in which less than $\log{d}$ vertices are blue, $v_{(i, 2)}$ is blue, and vertices $v_{(i, 1)}$ and $v_{(i + 1, 2)}$ are white, the expected amount of time until at least $\sqrt[4]{d}$ vertices are blue or vertex $v_{(i, 1)}$ is blue is $\Omega(\log \log d)$. 

\end{proof}

We now finish the proof of \ref{dregular2}. 

\begin{proof}[Proof of Lemma \ref{dregular2}]

In order for all vertices in $G$ to become blue, the process described in Lemma \ref{complete} must occur at least $\frac{n}{2(d+1)}-1$ times. The expected amount of time for the process to occur once is $\Omega(\log \log d)$. Thus for fixed $d$, we have $\ept(G) = \Omega\bigp{\frac{\log \log d}{d} \cdot n}$.
\end{proof}

The proof of Theorem \ref{dregular} follows by Lemma \ref{dregular1} and Lemma \ref{dregular2}.

\section{Hypercube graphs}

In this section, we prove Theorem \ref{hypercube}. We prove the following key lemma:

\begin{lemma}\label{hypercube1}
Let $G$ be an $n$-dimensional hypercube graph. At a given time in the probabilistic zero forcing process, let $S$ denote the set of blue vertices in $G$, and let $0 \le k < n-1$ be the integer such that $2^k \le |S| < 2^{k+1}$. Then the expected amount of time steps to go from $\abs{S}$ blue vertices to at least $2^{k+1}$ blue vertices is at most $131n \bigp{\frac{1}{n-k-1}}$.
\end{lemma}

\begin{proof}
By the Edge Isoperimetric Inequality, there are at least $|S|(n - \log_2{|S|}) \ge 2^k(n - k - 1)$ edges between a white and a blue vertex. A white vertex that is connected to $d$ blue vertices has a probability of at least $1 - (\frac{n - 1}{n})^{d}$ of becoming colored blue, since, for each of these $d$ edges, the probability that propagation does not occur on that edge is at most $1 - \frac{1}{n} = \frac{n-1}{n}$. 

Let $c_i$ be the number of white vertices in $G$ that are connected to exactly $i$ blue vertices. Then note that the expected number of white vertices colored blue is at least the expected number of successes in $\displaystyle \sum_{i=1}^n c_i$ independent trials in which $c_i$ trials have probability $1 - \bigp{\frac{n - 1}{n}}^{i}$ of success. Among positive integers $i\in [n]$, $\frac{\bigp{1 - \bigp{\frac{n - 1}{n}}^{i}}}{i}$ is minimized when $i = n$ and maximized when $i = 1$. Hence the expected number of successes among these $\displaystyle \sum_{i=1}^n c_i$ independent trials, which we denote by $\mu$, satisfies
\begin{align*}
        \mu & =  \sum_{i = 1}^{n} c_i\bigp{1 - \bigp{\frac{n - 1}{n}}^{i}} \\&= \sum_{i = 1}^{n} ic_i \cdot \frac{\bigp{1 - \bigp{\frac{n - 1}{n}}^{i}}}{i} \\ & \le \frac{\bigp{1 - \bigp{\frac{n - 1}{n}}}}{1} \bigp{\sum_{i = 1}^{n} i c_i} \\ & \le \frac{2^{k+1}}{n}.
\end{align*}
and
\begin{align*}
        \mu & = \sum_{i = 1}^{n} c_i\bigp{1 - \bigp{\frac{n - 1}{n}}^{i}} \\ & =  \sum_{i = 1}^{n} ic_i \cdot \frac{\bigp{1 - \bigp{\frac{n - 1}{n}}^{i}}}{i} \\ & \ge \frac{\bigp{1 - \bigp{\frac{n - 1}{n}}^{n}}}{n} \bigp{\sum_{i = 1}^{n} i c_i} \\ & \ge \frac{\bigp{1 - \bigp{\frac{n - 1}{n}}^{n}}}{n}\cdot 2^k\cdot (n-k-1) \\ & \ge \frac{e-1}{e} \cdot 2^k\cdot \bigp{\frac{n-k-1}{n}} \\&\ge \frac{e-1}{2e}.
\end{align*}
for $n\ge 1$ by the Edge Isoperimetric Inequality, noting that $\displaystyle \sum_{i = 1}^{n} ic_i$ is equal to the number of edges between a blue vertex and a white vertex. While there are $2^k \le \abs{S} < 2^{k+1}$ blue vertices, the probability that fewer than $\frac12 \mu$ white vertices are colored blue is at most $e^{-\frac18 \mu} \le e^{-\frac{e-1}{16e}} < \frac{25}{26}$. Thus, the
probability that at least $\frac12 \mu$ white vertices are colored blue is at least $\frac{1}{26}$.

The probability that there remain between $2^k$ and $2^{k+1}$ blue vertices in $52 \frac{2^k }{\mu/2}$ units of time is at most the probability that among $52 \frac{2^k }{\mu/2}$ events that each independently happen with probability $\frac{1}{26}$, fewer than $\frac{2^k }{\mu/2}$ occur. By the Chernoff bound, this probability is at most
\begin{equation*}
    \begin{split}
    p & =  e^{-\frac12 (\frac12)^2 \bigp{2\frac{2^k }{\mu/2}}} \\ & =  e^{-\frac12 \bigp{\frac{2^{k}}{\mu}}} \\ & \le e^{-n} \\ & \le 0.37
    \end{split}
\end{equation*}

Therefore, if $2^k \le \abs{S} < 2^{k+1}$, with probability at least $0.63$, after $52 \frac{2^k}{\mu/2}$ units of time the number of blue vertices will be greater than $2^{k+1}$. We conclude that the expected time to go from $\abs{S}$ blue vertices to at least $2^{k+1}$ blue vertices is at most
\begin{align*}
    \frac{1}{0.63} \cdot 52 \cdot \frac{2^k}{\mu/2} &\le \frac{1}{0.63} \cdot 52 \cdot \frac{e}{e-1} \cdot \frac{n}{n-k-1}\\&\le 131\frac{n}{n-k-1}
\end{align*}

\end{proof}

\begin{lemma}\label{hypercube2}
Let $G$ be an $n$-dimensional hypercube graph. At a given time in the probabilistic zero white vertices in $G$, and let $0 \le k < n-1$ be the integer such that $2^k < |S| \leq 2^{k+1}$. Then the expected amount of time steps to go from $\abs{S}$ white vertices to at most $2^{k}$ white vertices is at most $131n \bigp{\frac{1}{n-k-1}}$.
\end{lemma}

\begin{proof}

The proof proceeds similarly to the proof of Lemma \ref{hypercube1}. By the Edge Isoperimetric Inequality, there are at least $|S|(n - \log_2{|S|}) \ge 2^k(n - k - 1)$ edges between a white and a blue vertex.

Let $c_i$ be the number of white vertices in $G$ that are connected to exactly $i$ blue vertices. Then note that the expected number of white vertices colored blue is at least the expected number of successes in $\displaystyle \sum_{i = 1}^{n} c_i$ independent trials in which $c_i$ trials have probability $1 - (\frac{n - 1}{n})^{i}$ of success. The expected number of successes among these $\displaystyle \sum_{i = 1}^{n} c_i$ independent trials, which we denote by $\mu$, satisfies
\begin{equation*}
    \begin{split}
        \frac{e-1}{2e} \le \mu \le \frac{2^{k+1}}{n}
    \end{split}
\end{equation*}
for $n\ge 1$. While there are $2^k < \abs{S} \le 2^{k+1}$ white vertices, the probability that fewer than $\frac12 \mu$ white vertices are colored blue is at most $e^{-\frac18 \mu} < \frac{25}{26}$. Thus, the
probability that at least $\frac12 \mu$ white vertices are colored blue is at least $\frac{1}{26}$. The probability that there remain between $2^k$ and $2^{k+1}$ blue vertices in $52 \frac{2^k }{\mu/2}$ units of time is at most the probability that among $52 \frac{2^k }{\mu/2}$ events that each independently happen with probability $\frac{1}{26}$, fewer than $\frac{2^k }{\mu/2}$ occur. This probability is at most $0.37$ as in Lemma \ref{hypercube1}.

Therefore, if $2^k < \abs{S} \leq 2^{k+1}$, with probability at least $0.63$, after $52 \frac{2^k}{\mu/2}$ units of time the number of white vertices will be at most $2^{k}$. We conclude that the expected time to go from $\abs{S}$ white vertices to at most  $2^{k}$ white vertices is at most $$\frac{1}{0.63} \cdot 52 \cdot \frac{2^k}{\mu/2} \le 131\frac{n}{n-k-1}.$$

\end{proof}

We are finally ready to prove Theorem \ref{hypercube}.

\begin{proof}[Proof of Theorem \ref{hypercube}]
By Linearity of Expectation, the expected amount of time, starting from $1$ initial blue vertex, for there to be at least $2^{n - 1}$ blue vertices is at most the sum over $i=0,1,\ldots,n-2$ of the expected amount of times to go from $2^k \le \abs{S} < 2^{k+1}$ blue vertices to at least $2^{k+1}$ blue vertices. By Lemma \ref{hypercube1}, this sum is at most
\begin{equation*}
\sum_{k = 0}^{n - 2} 131n \cdot \frac{1}{n-k-1} \\ \le 131n \cdot (1 + \log n).
\end{equation*}
Once there at least $2^{n-1}$ blue vertices, there are at most $2^{n-1}$ white vertices. The expected amount of time to go from at most $2^{n - 1}$ white vertices to at most one white vertex is, by Lemma \ref{hypercube2}, at most 
\begin{equation*}
\sum_{k = 0}^{n - 2} 131n \cdot \frac{1}{n-k-1} \\ \le 131n \cdot (1 + \log n).
\end{equation*}
When there is at most one white vertex remaining, this vertex has a probability of $$1 - \bigp{\frac{n - 1}{n}}^{n} \ge 1 - \frac{1}{e}$$ of being colored blue at any time step, so the expected amount of time steps until it is blue is $\frac{e}{e-1}$. We conclude that the expected propagation time of the entire hypercube is $O(n \log n)$.
\end{proof}

\begin{remark}
Experimentally, for this family of graphs, the expected propagation time appears to approximate $n + 0.8$ for small $n$.
\end{remark}

For each $1 \le n \le 16$, we run a program to simulate the probabilistic zero-forcing process $1000$ times on a hypercube graph with dimension $n$ and $2^n$ vertices, starting from a single blue vertex. Figure \ref{figure-2} shows the average propagation time over $1000$ trials for each value of $n$.

\begin{figure}[h]
    \centering

\begin{tabular}{ |c|c| } 
 \hline
 $n$ & $\ept$ \\ \hline
 1 & 1.00 \\
 \hline
 2 & 2.32 \\
  \hline
 3 & 3.51 \\
  \hline
 4 & 4.68 \\
  \hline
 5 & 5.78 \\
  \hline
 6 & 6.79 \\
  \hline
 7 & 7.78 \\
  \hline
 8 & 8.79 \\
  \hline
 9 & 9.78 \\
  \hline
 10 & 10.81 \\
  \hline
 11 & 11.78 \\
  \hline
 12 & 12.82 \\
  \hline
 13 & 13.88 \\
  \hline
 14 & 14.80 \\
  \hline
 15 & 15.80 \\
  \hline
 16 & 16.79 \\
  \hline
\end{tabular}
    \caption{Simulated expected propagation times of hypercube graphs of dimension $n$, $1 \le n \le 16$.}
    \label{figure-2}
\end{figure}


\begin{thebibliography}{99}

\bibitem{1}
 F. Barioli, W. Barrett, S.M. Fallat, H.T. Hall, L. Hogben, B. Shader, P. van den Driessche, and H. van der
Holst. \emph{Zero forcing parameters and minimum rank problems}. Linear Algebra Appl. 433 (2010) 401-411.

\bibitem{3}
 K.F. Benson, D. Ferrero, M. Flagg, V. Furst, L. Hogben, V. Vasilevska, and B. Wissman. \emph{Zero forcing
and power domination for graph products}. Australasian J . Combinatorics 70 (2018), 221-235.

\bibitem{2}
A. Berman, S. Friedland, L. Hogben, U.G. Rothblum, and B. Shader. \emph{An upper bound for the minimum
rank of a graph}. Linear Algebra Appl. 429 (2008) 1629-1638.

\bibitem{Finite}
Y. Chan, E. Curl, J. Geneson, L. Hogben, K. Liu, I. Odegard, and M. Ross. \emph{Using Markov chains to determine expected propagation time for probabilistic zero-forcing}. arXiv:1906.11083.

\bibitem{5e}
R. Daknama, K. Panagiotou, and S. Reisser. \emph{Robustness of randomized rumour spreading}.
arXiv:1902.07618, 2019.

\bibitem{7}
C.J. Edholm, L. Hogben, M. Hyunh, J. LaGrange, and D.D. Row. \emph{Vertex and edge spread of zero forcing
number, maximum nullity, and minimum rank of a graph}. Linear Algebra Appl. 436 (2012) 4352-4372.

\bibitem{english}
S. English, C. MacRury, and P. Pralat. \emph{Probabilistic zero forcing on random graphs}.
European J. Combin. 91 (2021), 103207.

\bibitem{8}
S. Fallat, S. Severini, M. Young. \emph{AIM workshop: Zero forcing and its variants}. American Institute of
Mathematics, Jan. 30-Feb 3, 2017.

\bibitem{GH}
J. Geneson and L. Hogben. \emph{Propagation time for probabilistic zero-forcing}. arXiv:1812.10476.

\bibitem{Doob}
G.R. Grimmett and D.R. Stirzaker. \emph{Probability and Random Processes (3rd ed.)}. Oxford University Press.
(2001) 491-495.

\bibitem{11}
L. Hogben, M. Huynh, N. Kingsley, S. Meyer, S. Walker, and M. Young, \emph{Propagation time for zero forcing
on a graph}. Discrete Appl. Math. 160 (2012) 1994-2005.

\bibitem{12}
L. Huang, G.J. Chang, and H. Yeh, \emph{On minimum rank and zero forcing sets of a graph}. Linear Algebra Appl. 432 (2010) 2961-2973.

\bibitem{13}
C.X. Kang and E. Yi. \emph{Probabilistic zero forcing in graphs}. Bull. Inst. Combin. Appl. 67 (2013), 9-16.
 
\bibitem{13e}
A. Mehrabian and A. Pourmiri. \emph{Randomized rumor spreading in poorly connected small-world
networks}. Random Structures Algorithms, 49(1):185–208, 2016.

\bibitem{Sun}
S. Narayanan and A. Sun. \emph{Bounds on expected propagation time of probabilistic zero-forcing}.
arXiv:1909.04482

\bibitem{15}
B. Yang. \emph{Fast-mixed searching and related problems on graphs}. Theoret. Comput. Sci. 507 (2013), 100-113.

\end{thebibliography}
\end{document}